\documentclass[11pt,bezier]{article}
\usepackage{amsmath,amssymb,amsfonts,euscript,graphicx}

\title{{\bf Rings Whose Annihilating-Ideal Graphs Have Positive Genus\thanks {The research of
 the second  author was in part supported by a grant from IPM (No.
 89160031)}}}
\author{{{\bf F. Aliniaeifard${^{\rm a}}$ and   M. Behboodi${^{\rm b}}$}}\thanks
{ Corresponding author}\\
  {\small{${^{\rm a,b}}$Department of Mathematical Sciences, Isfahan University of Technology}}\vspace{-1mm}\\ {\small{
Isfahan,
Iran, 84156-83111}}\\
{\small{$^{\rm b}$School of Mathematics,  Institute for Research in Fundamental Sciences (IPM), }} \\
 {\small{Tehran  Iran, 19395-5746}}\\
\footnotesize{${^{\rm a}}$f.aliniaeifard@math.iut.ir}\vspace{-1mm}\\
\footnotesize{${^{\rm b}}$mbehbood@cc.iut.ac.ir}}

\def\be{\begin{enumerate}}
\def\ee{\end{enumerate}}

\newtheorem{ttheo}{Theorem}[section]
\newtheorem{ccoro}[ttheo]{Corollary}
\newtheorem{llem}[ttheo]{Lemma}
\newtheorem{eexam}[ttheo]{Example}
\newtheorem{rrem}[ttheo]{Remark}
\newtheorem{ppro}[ttheo]{Proposition}

\newenvironment{pproof}{\noindent{\bf Proof. }}{}

\date{}
\begin{document}
  \maketitle
\begin{abstract}
{ Let $R$ be a commutative ring and ${\Bbb{A}}(R)$ be the set of
ideals with non-zero annihilators. The annihilating-ideal graph of
$R$ is defined as the graph ${\Bbb{AG}}(R)$ with the vertex set
${\Bbb{A}}(R)^*={\Bbb{A}}\setminus\{(0)\}$ and two distinct
vertices $I$ and $J$ are adjacent if and only if $IJ=(0)$.  We
investigate commutative
   rings $R$  whose annihilating-ideal graphs have positive genus $\gamma(\Bbb{AG}(R))$. It
   is shown that if $R$ is an Artinian ring such that  $\gamma(\Bbb{AG}(R))<\infty$,
   then $R$ has finitely many ideals or $(R,\mathfrak{m})$ is a Gorenstein ring with maximal ideal $\mathfrak{m}$ and
   ${\rm v.dim}_{R/{\mathfrak{m}}}{\mathfrak{m}}/{\mathfrak{m}}^{2}=2$.
   Also,  for any two integers $g\geq 0$ and $q>0$, there are finitely many isomorphism classes of Artinian
    rings $R$ satisfying  the  conditions:
    (i) $\gamma(\Bbb{AG}(R))=g$  and (ii)
   $|R/{\mathfrak{m}}| \leq q$ for every maximal ideal ${\mathfrak{m}}$ of $R$. Also,
    it is shown that if $R$ is a non-domain  Noetherian local ring such that
        $\gamma(\Bbb{AG}(R))<\infty$,
     then  either $R$ is a Gorenstein ring or  $R$ is an Artinian ring with finitely many ideals.\vspace{3mm}\\
  {\footnotesize{\it\bf Key Words:} Commutative ring; annihilating-ideal graph;
  genus of a   graph.}\\
  {\footnotesize{\bf 2010  Mathematics Subject
  Classification:}   05C10; 13E05; 13E10;  13M05; 16P60.}}
\end{abstract}

  \section{\large\bf Introduction}
 The study of algebraic structures, using the properties of
 graphs, became an exciting research topic in the last twenty
 years, leading to many fascinating results and questions. There
 are many papers on assigning a graph to a ring, for instance see,
 \cite{and-liv}, \cite{and-nas}, \cite{beck}, \cite{beh-rak1} and \cite{beh-rak2}.  Throughout this paper, $R$
  denotes a commutative  ring with $1\neq 0$. Let $X$ be
either an element or a subset of $R$. The annihilator of $X$ is
the  ideal Ann$(X)=\{a\in R~:~aX=0\}$. We note that an element
$a\in R$ is called a
 zero-divisor in $R$ if  Ann$(a)\neq \{0\}$ (i.e.,  $ab=0$ for some nonzero element $b\in R$).
  For any subset $Y$ of $R$, we denote $|Y|$ for the cardinality of $Y$ and let
$Y^*=Y\setminus\{0\}$. Also, we denote the finite field of order
$q$ by ${\Bbb{F}}_q$. Let $V$ be a vector space over the field
$\Bbb{F}$. We use the notation $\it {{\rm v.dim}}_{\Bbb{F}}(V)$ to
denote the dimension of $V$ over the field $\Bbb{F}$.

   Let $S_{k}$ denote the sphere with $k$ handles, where $k$ is a non-negative integer, that
is, $S_{k}$ is an oriented surface of genus $k$. The {\it genus}
of a graph $G$, denoted $\gamma(G)$, is the minimal integer $n$
such that the graph can be embedded in $S_{n}$. For details on the
notion of embedding a graph in a surface, see, e.g., \cite[Chapter
6]{white}. Intuitively, $G$ is embedded in a surface if it can be
drawn in the surface so that its edges intersect only at their
common vertices. A genus $0$ graph is called a {\it planar graph}
and a genus $1$ graph is called a {\it toroidal graph}. An
infinite graph $G$ has infinite genus ($\gamma (G)=\infty$), if,
for every natural number $n$, there exists a finite subgraph
$G_{n}$ of $G$ such that $\gamma(G_{n})=n $. We note here that if
$H$ is a subgraph of a graph $G$, then $\gamma(H) \leq \gamma(G)$.
Let $K_n$ denote the complete graph on $n$ vertices; that is,
$K_n$ has vertex set $V$ with $|V|=n$ and
$a-\hspace{-1.5mm}-\hspace{-1.5mm}- b$ is an edge for every
distinct $a,b\in V$. Let $K_{m,n}$ denote the complete bipartite
graph; that is, $K_{m,n}$ has vertex set $V$ consisting of the
disjoint union of two subsets, $V_{1}$ and $V_2$, such that
$|V_{1}|= m$ and $|V_{2}|=n$, and
$a-\hspace{-1.5mm}-\hspace{-1.5mm}- b$ is an edge if and only if
$a \in V_{1}$ and $b\in V_{2}$. We may sometimes write
$K_{|V_{1}|,|V_{2}|}$ to denote the complete bipartite graph with
vertex sets $V_{1}$ and $V_{2}$. Note that $K_{m,n} = K_{n,m}$. It
is well known that
$$\gamma(K_{n})=\lceil \frac{(n-3)(n-4)}{12}\rceil ~~~{\rm for~ all}~ n \geq 3.$$
$$\gamma(K_{m,n})= \lceil \frac{(n-2)(m-2)}{4} \rceil ~~~{\rm for~all}~n \geq 2~ and~ m \geq 2.$$
(see \cite{ring-you}; \cite{ring1}, respectively). For a graph
$G$, the degree of a vertex $\it{v}$ of $G$ is the number of edges
of $G$ incident with $\it{v}$.

Let $R$ be a ring.    We call an
 ideal $I$ of $R$, an {\it annihilating-ideal} if there exists a
 non-zero ideal $J$ of $R$ such that $IJ=(0)$.    We denote  $Z(R)$ the sets of zero-divisors, $\Bbb{A}(R)$ for the set of
  all annihilating-ideals of $R$, $J(R)$ for the  Jacobson radical of $R$,  and  for an ideal $J$ of $R$, we denote
 $\Bbb{I}(J)$ for the set $\{I~:~I~{\rm be~an~ideal~of}~R~{\rm and}~I\subseteq J\}$.
 Also,  by the {\it annihilating-ideal
   graph} $\Bbb{AG}(R)$ of $R$ we mean the graph with vertices
   $\Bbb{AG}(R)^{*} = \Bbb{A}(R) \setminus \{(0)\}$ such that
   there is an (undirect) edge between vertices $I$ and $J$ if and
   only if $I\neq J$ and $IJ=(0)$. Thus $\Bbb{AG}(R)$ is an empty
   graph if and only if $R$ is an integral domain. Recently, the
   idea of annihilating-ideal graph of commutative rings were
   introduced by Behboodi and Rakeei in \cite{beh-rak1},
   \cite{beh-rak2} and by Behboodi et al., in \cite{aal}. The {\it zero-divisor graph} of $R$, $\Gamma(R)$,
   is an undirect graph with the vertex set $Z(R)\setminus \{0\}$ and two distinct
   vertices $x,y$ are adjacent if and only if $xy=(0)$. There are
   some papers which investigate genus of a zero-divisor graph,
   for instance see \cite{wick1}, \cite{wick2} and \cite{wang}. In
   particular in \cite[Theorem 2]{wick2}, it was shown that for any positive integer $g$,
    there are finitely many finite commutative rings whose zero-divisor graph has genus $g$.

In the present paper we investigate
   rings whose annihilating-ideal graphs have positive genus. In
   Theorem \ref{2.7}, it is shown that if $R$ is an Artinian ring and $\gamma(\Bbb{AG}(R)) <
   \infty$, then $R$ has finitely many ideals or $(R,\mathfrak{m})$ is
   Gornestein ring with $dim_{R/{\mathfrak{m}}}{\mathfrak{m}}/{\mathfrak{m}}^{2}=2$
    (a Noetherian local ring $(R,\mathfrak{m})$ is called {\it Gorenstein} if
    ${{\rm v.dim}}_{R/{\mathfrak{m}}}({\rm Ann}({\mathfrak{m}}))=1$).  Also, it is shown
    that,  for two
integers $g\geq 0$ and $q>0$, there are finitely many Artinian
rings $R$ such that
    (i) $\gamma(\Bbb{AG}(R))=g$  and (ii)
   $|R/{\mathfrak{m}}| \leq q$ for every maximal ideal $\mathfrak{m}$ of $R$
   (see Theorem \ref{2.10} and Corollary \ref{2.11}).
    Also, the genus of annihilating-ideal graphs
   of Noetherian  rings are studied. It is shown that if $R$ is a
    Noetherian local ring and $\gamma(\Bbb{AG}(R))<\infty$, then either $R$ is a domain or
     all non-trivial ideals of $R$ are vertices of $\Bbb{AG}(R)$ (see Proposition \ref{3.4}).
     Also if  $R$ is a Noetherian ring such that all non-trivial ideals of
$R$ are vertices of $\Bbb{AG}(R)$ and
$\gamma(\Bbb{AG}(R))<\infty$,  then either $R$ is  a Gorenstein
ring or $R$ is an Artinian ring with finitely many ideals (see
Theorem \ref{3.5}). Finally,  it is shown that if $R$ is a
non-domain Noetherian local ring such that
$\gamma(\Bbb{AG}(R))<\infty$,
 then either $R$ is  a Gorenstein ring or $R$ is an Artinian ring with finitely many ideals
  (see Corollary \ref{3.6}).

\section{\large\bf The Genus of the Annihilating-Ideal Graphs of  Artinian Rings}

  The following
useful lemma and remark   will be used frequently in this paper.

\begin{llem}\label{2.1} {\rm \cite[Proposition 1.3]{beh-rak1}} {\it Let $R$ be an Artinian ring.
Then every nonzero proper ideal $I$ is a vertex of $\Bbb{AG}(R)$}.
\end{llem}

\begin{rrem}\label{2.2} {\rm It is well known that if $V$ is a vector space over an
infinite field $\Bbb{F}$, then $V$ can not be the union of
finitely many proper subspaces; see for example
\cite[p.283]{hal}.}
\end{rrem}

A local Artinian principal ring is called a {\it special principal
ring}
  and has an extremely simple ideal structure: there are only finitely many ideals,
   each of which is a power of the maximal ideal.

\begin{llem}\label{2.3}
 Let $(R,\mathfrak{m})$ be a local ring with ${\mathfrak{m}}^t=(0)$. If for a positive integer
$n$, ${\it {\rm
v.dim}_{R/{\mathfrak{m}}}}({\mathfrak{m}}^{n}/{\mathfrak{m}}^{n+1})
=1$ and ${\mathfrak{m}}^{n}$ is a finitely generated $R$-module,
 Then $\Bbb{I}({\mathfrak{m}}^{n})=\{{\mathfrak{m}}^i~:~t\leq i\leq n\}$. Moreover
 If $n=1$, then $R$ is a  special principal
ring.
 \end{llem}

\begin{proof}   Since ${\rm v.dim}_{R/{\mathfrak{m}}}({\mathfrak{m}}^{n}/{\mathfrak{m}}^{n+1})=1$, by using Nakayama's lemma, ${\mathfrak{m}}^n=Rx$ for some
$x\in R$. Now,  let $I$ be a nonzero ideal of $R$ such that
$I\subseteq {\mathfrak{m}}^{n}$. Since ${\mathfrak{m}}^t=(0)$,
there exists a natural number $i\geq n$ such that $I \subseteq
{\mathfrak{m}}^{i}$ and $I \nsubseteq {\mathfrak{m}}^{i+1}$. Let
$a \in I \setminus {\mathfrak{m}}^{i+1}$. We have $a= bx^{i}$, for
some $b \in R$. If $b \in m$, then $a \in {\mathfrak{m}}^{i+1}$, a
contradiction. Thus $b$ is unit. Hence $x^{i} \in I$. This implies
that $I = (x^{i})$, as desired. Thus
$\Bbb{I}({\mathfrak{m}}^{n})=\{{\mathfrak{m}}^i~:~t\leq i\leq
n\}$. Clearly, if $n=1$, then $R$ is a  special principal ring.
\hfill $\square$
\end{proof}

\begin{llem}\label{2.4}
Let $(R,\mathfrak{m})$ be an Artinian local ring with
$|R/{\mathfrak{m}}|=\infty$, ${\mathfrak{m}}^{3}=(0)$ and ${\rm
v.dim}_{R/{\mathfrak{m}}} {\rm Ann}({\mathfrak{m}})=1$. If
$\gamma(\Bbb{AG}(R))<\infty$, then ${\rm
v.dim}_{R/{\mathfrak{m}}}{\mathfrak{m}}/{\mathfrak{m}}^{2}\leq 2$
\end{llem}

\begin{proof}
First we assume that  ${\it {\rm v.dim}_{R/{\mathfrak{m}}}}
({\mathfrak{m}}/{\mathfrak{m}}^{2}) =3$. Then there exist
infinitely many subspaces with dimension one. Let
$Rx/{\mathfrak{m}}^{2}$ and $Ry/{\mathfrak{m}}^{2}$ be two
distinct subspaces of dimension one of
${\mathfrak{m}}/{\mathfrak{m}}^{2}$. Since $Rx\cong R/{\rm
Ann}(x)$ and $Rx$ has only one
 nonzero proper $R$-submodule,  ${\mathfrak{m}}/{\rm{Ann}}(x)$ is the only  nonzero proper ideal of $R/{\rm Ann}(x)$.
 Therefore,  ${\rm  v.dim}_{R/{\mathfrak{m}}}{\rm Ann}(x)=2$. Similarly, ${\rm v.dim}_{R/{\mathfrak{m}}}{\rm Ann}(y)=2$.
  Therefore,  $|\Bbb{I}({\rm Ann}(x))|=|\Bbb{I}({\rm Ann}(y))|=\infty$.
  If  ${\rm Ann}(x)={\rm Ann}(y)$, then since  $(Rx){\rm Ann}(x)=(0)$,
  $(Ry){\rm Ann}(x)=(0)$ and ${\mathfrak{m}}^{2}{\rm Ann}(x)=(0)$, we conclude that $K_{|\Bbb{I}({\rm
  Ann}(x))|,3}$  is a subgraph of $\Bbb{AG}(R)$. Thus by the formula for genus of  complete bipartite
   graph, $\gamma(\Bbb{AG}(R))=\infty$, a contradiction. Thus we may assume
    that ${\rm Ann}(x)\neq {\rm Ann}(y)$. If ${\rm  Ann}(x)\cap {\rm Ann}(y)=(0)$, then
     $({\rm Ann}(x))({\rm Ann}(y))=(0)$ and so  $K_{|\Bbb{I}({\rm Ann}(x))|}$
 is a subgraph of $\Bbb{AG}(R)$, and hence, by the formula for complete
 graphs $\gamma(\Bbb{AG}(R))=\infty$, a contradiction. Therefore,  ${\rm  Ann}(x)\cap {\rm Ann}(y)\neq(0)$ and so
  ${\rm v.dim}_{R/{\mathfrak{m}}}({\rm Ann}(x)\cap {\rm Ann}(y))/{\mathfrak{m}}^{2}=1$.\\
  Suppose that $(Rx)(Ry)\neq (0)$. Since ${\rm v.dim}_{R/{\mathfrak{m}}}({\rm Ann}(x)\cap {\rm Ann}(y))/{\mathfrak{m}}^{2}=1$, there exists ideal $K$ such
 that $Rx-\hspace{-2mm}-\hspace{-2mm}-K-\hspace{-2mm}-\hspace{-2mm}-Ry$. Let $I_{1}$ be an ideal such that $I_{1}\in \Bbb{I}({\rm Ann}(x))\setminus \{K,Ry\}$  and ${\rm
 v.dim}_{R/{\mathfrak{m}}}a_{1}/{\mathfrak{m}}^{2}=1$. Let $J_{1}$ be an ideal such that $J_{1}\in \Bbb{I}({\rm Ann}(x))\setminus\{Rx,Ry,K,I_{1}, {\rm
 Ann}(x)\cap {\rm Ann}(y)\}$ such that ${\rm v.dim}_{R/{\mathfrak{m}}}J_{1}/{\mathfrak{m}}^{2}=1$. Let $K_{1}={\rm Ann}(I_{1})\cap {\rm
 Ann}(J_{1})$. Therefore,  $Rx-\hspace{-2mm}-\hspace{-2mm}-I_{1}-\hspace{-2mm}-\hspace{-2mm}-K_{1}-\hspace{-2mm}-\hspace{-2mm}-J_{1}-\hspace{-2mm}-\hspace{-2mm}-Ry$
 is a path between $Rx$ and $Ry$. Now, let $I_{n}\in \Bbb{I}({\rm Ann}(x))\setminus
 \{I_{i},J_{i},K_{i},Rx,Ry, {\rm Ann}(x)\cap {\rm Ann}(y),{\rm Ann}(x)\cap {\rm Ann}(K_{i}), i=1,2,..n-1\}$ such that ${\rm
 v.dim}_{R/{\mathfrak{m}}}I_{n}/{\mathfrak{m}}^{2}=1$ and $J_{n}\in \Bbb{I}({\rm Ann}(x))\setminus \{I_{i},J_{i-1},K_{i-1},I_{n},Rx,Ry, {\rm
 Ann}(x)\cap {\rm Ann}(y),{\rm Ann}(x)\cap {\rm Ann}(K_{i-1}), i=1,2,..n\}$ such that ${\rm v.dim}_{R/{\mathfrak{m}}}J_{n}/{\mathfrak{m}}^{2}=1$.
  Suppose that  $K_{n}={\rm Ann}(x_{n})\cap {\rm Ann}(y_{n})$, then  $Rx-\hspace{-2mm}-\hspace{-2mm}-I_{n}-\hspace{-2mm}-\hspace{-2mm}-K_{n}-\hspace{-2mm}-\hspace{-2mm}-J_{n}-\hspace{-2mm}-\hspace{-2mm}-Ry$ is a path between $Rx$ and $Ry$.
 Therefore,  there exists infinitely many path  between $Rx$ and $Ry$.
Thus either $(Rx)(Ry)=(0)$ or there exist infinitely many path between $Rx$ and $Ry$.\\
  Since $Rx/{\mathfrak{m}}^{2}$ and $Ry/{\mathfrak{m}}^{2}$ are two arbitrary distinct
  one dimensional
 $R/{\mathfrak{m}}$-subspaces of ${\rm Ann}(x)$ and there are infinitely many
 subspaces of dimension one of ${\rm Ann}(x)$, one can easily see
 that $\gamma(\Bbb{AG}(R))=\infty$, a contradiction.\\
Now, we assume that  ${\it {\rm v.dim}_{R/{\mathfrak{m}}}}
({\mathfrak{m}}/{\mathfrak{m}}^{2}) >3$. Suppose that
$\{x_{1}+{\mathfrak{m}}^{2}, x_{2}+{\mathfrak{m}}^{2},
x_{3}+{\mathfrak{m}}^{2},...,x_{n}+{\mathfrak{m}}^{2}\}$ is  a
basis for ${\mathfrak{m}}/{\mathfrak{m}}^{2}$ over
$R/{\mathfrak{m}}$. Since $Rx_{1}\cong R/{\rm Ann}(x_{1})$ and
$Rx_{1}$ has one $R$-submodule, ${\it {\rm
v.dim}_{R/{\mathfrak{m}}}} ({\rm Ann}(x_{1})/{\mathfrak{m}}^{2})
=n-1$. Similarly ${\it {\rm v.dim}_{R/{\mathfrak{m}}}} ({\rm
Ann}(x_{2})/{\mathfrak{m}}^{2}) =n-1$. Therefore,  ${\it {\rm
v.dim}_{R/{\mathfrak{m}}}} ({\rm Ann}(x_{2})\cap {\rm
Ann}(x_{1}))/{\mathfrak{m}}^{2}) =n-2$. Since $(Rx_{1})({\rm
Ann}(x_{2})\cap {\rm Ann}(x_{1}))=(0)$, $(Rx_{2})({\rm
Ann}(x_{2})\cap {\rm Ann}(x_{1}))=(0)$ and
${\mathfrak{m}}^{2}({\rm Ann}(x_{2})\cap {\rm Ann}(x_{1}))=(0)$,
we conclude that  $K_{|\Bbb{I}(({\rm Ann}(x_{2})\cap {\rm
Ann}(x_{1})))|,3}$ is a subgraph of $\gamma(\Bbb{AG}(R))$. Thus by
the formula for genus of complete bipartite graphs
$\gamma(\Bbb{AG}(R)) = \infty$, a contradiction. \hfill$\square$
\end{proof}

\begin{ttheo}\label{2.5}
 Let $(R,\mathfrak{m})$ be an Artinian local ring.
If $\gamma(\Bbb{AG}(R)) < \infty$, then $R$ has finitely many
ideals or $R$ is a Gorenstein ring with
$dim_{R/{\mathfrak{m}}}{\mathfrak{m}}/{\mathfrak{m}}^{2}=2$.
\end{ttheo}

\begin{proof} If $R$ is a field, then  $R$ has finitely many
ideals. Thus we can assume that $R$ is not a field. If
$|R/{\mathfrak{m}}|<\infty$, then one can easily see that $R$ is a
finite ring and so $R$ has finitely many ideals. Thus we can
assume that $|R/{\mathfrak{m}}|=\infty$ and $\gamma(\Bbb{AG}(R))=
g$ for an integer $g \geq 0$.
The proof now proceeds by cases:\\
 {\it Case} 1: ${\it {\rm v.dim}_{R/{\mathfrak{m}}}}({\rm Ann}({\mathfrak{m}})) \geq
 2$. Then $|{\Bbb{I}}(\rm{Ann}({\mathfrak{m}}))|=|\Bbb{I}({\mathfrak{m}})|=\infty$ and since  ${\mathfrak{m}}{\rm Ann}({\mathfrak{m}})=(0)$,
  $K_{|\Bbb{I}({\mathfrak{m}})|-4,3}$ is a subgraph of $\Bbb{AG}(R)$.
 Hence, by the formula for genus of complete bipartite graphs
  $\lceil
  (|\Bbb{I}({\mathfrak{m}})|-6)/4 \rceil \leq g$ and so $|\Bbb{I}({\mathfrak{m}})|\leq 4g+6$, a contradiction.\\
 {\it Case} 2: ${\it {\rm v.dim}_{R/{\mathfrak{m}}}} ({\rm Ann}({\mathfrak{m}})) =1$.
 Since $R$ is an Artinian ring, there exists positive integer $t$
 such that
 ${\mathfrak{m}}^{t+1}=(0)$ and ${\mathfrak{m}}^{t} \neq (0)$.\\
 {\it Subcase} 1:  $t=1$, i.e., ${\mathfrak{m}}^{2}=(0)$. Then
 $K_{|\Bbb{I}({\mathfrak{m}})|-1}$ is a subgraph of $\Bbb{AG}(R)$ and so by the
 formula for genus of complete graphs,
$\lceil(|\Bbb{I}({\mathfrak{m}})|-6)/12\rceil\leq g$. Hence
$|\Bbb{I}({\mathfrak{m}})| \leq 12g+6$, .i.e., $R$ has at most
$12g+7$ ideals.\\
{\it Subcase} 2: $t=2$, i.e., ${\mathfrak{m}}^{2}\neq (0)$ and
${\mathfrak{m}}^{3}=(0)$. Then Lemma \ref{2.4} implies that ${\it
{\rm v.dim}_{R/{\mathfrak{m}}}}
({\mathfrak{m}}/{\mathfrak{m}}^{2})\leq 2$. If ${\it {\rm
v.dim}_{R/{\mathfrak{m}}}} ({\mathfrak{m}}/{\mathfrak{m}}^{2})=1$,
then Lemma \ref{2.3}
implies that $|\Bbb{I}({\mathfrak{m}})|=3$,  .i.e., $R$ has four ideals.\\
 {\it Subcase} 3:   $t \geq 3$.
 Since ${\mathfrak{m}}^{3}{\mathfrak{m}}^{t}={\mathfrak{m}}^{3}{\mathfrak{m}}^{t-1}={\mathfrak{m}}^{3}{\mathfrak{m}}^{t-2}=(0)$,
$K_{|\Bbb{I}({\mathfrak{m}}^{3})|,3}$ is a subgraph of
$\Bbb{AG}(R)$. So by the
 formula for genus of complete bipartite graphs,
$\lceil(|\Bbb{I}({\mathfrak{m}}^{3})|-6)/12\rceil\leq g$. Hence,
$|\Bbb{I}({\mathfrak{m}}^{3})| \leq 12g+6$. If ${\it {\rm
v.dim}_{R/{\mathfrak{m}}}}({\mathfrak{m}}^{t-1}/{\mathfrak{m}}^{t})\geq
2$, then Remark \ref{2.2} implies that
$|\Bbb{I}({\mathfrak{m}}^{t-1})|=\infty$.
 Since
${\mathfrak{m}}^{t-1}{\mathfrak{m}}^{t-1}=(0)$ and $t\geq 3$,
$K_{|\Bbb{I}({\mathfrak{m}}^{t-1})|-1}$ is a subgraph of
$\Bbb{AG}(R)$. Therefore,  $\gamma(\Bbb{AG}(R)) = \infty$, a
contradiction. Thus ${\it {\rm
v.dim}_{R/{\mathfrak{m}}}}({\mathfrak{m}}^{t-1}/{\mathfrak{m}}^{t})=
1$. Hence, by Lemma \ref{2.3}, there exists
$x \in {\mathfrak{m}}^{t-1}$ such that ${\mathfrak{m}}^{t-1}= Rx$.\\
Suppose that ${\rm {\rm
v.dim}}_{R/{\mathfrak{m}}}{\mathfrak{m}}/{\mathfrak{m}}^{2}=n\geq
3$. Since $Rx\cong R/{\rm Ann}(x)$ and $Rx$ has only one
 nonzero proper $R$-submodule,  ${\mathfrak{m}}/{\rm{Ann}}(x)$ is the only  nonzero proper ideal of $R/{\rm Ann}(x)$.
If ${\mathfrak{m}}^{2}\nsubseteq {\rm Ann}(y)$, then ${\rm
Ann}(x)+{\mathfrak{m}}^{2}={\mathfrak{m}}$, and  by Nakayama's
lemma, ${\rm Ann}(x)={\mathfrak{m}}$, a contradiction. Thus
${\mathfrak{m}}^{2}\subseteq {\rm Ann}(x)$ and since
${\mathfrak{m}}/{\rm Ann}(x)$ is the only nonzero proper ideal of
$R/{\rm Ann}(x)$, ${\rm v.dim}_{R/{\mathfrak{m}}}{\rm
Ann}(x)/{\mathfrak{m}}^{2}=n-1$. Let
$\{y_{1}+{\mathfrak{m}}^{2},y_{2}+{\mathfrak{m}}^{2},...,y_{n-1}+{\mathfrak{m}}^{2},y+{\mathfrak{m}}^{2}\}$
be a basis of ${\mathfrak{m}}/{\mathfrak{m}}^{2}$ such that
$\{y_{1}+{\mathfrak{m}}^{2},y_{2}+{\mathfrak{m}}^{2},...,y_{n-1}+{\mathfrak{m}}^{2}\}$
be a basis of
${\rm Ann}(x)/{\mathfrak{m}}^{2}$. Since ${\mathfrak{m}}^{t}$ is the only minimal ideal of $R$, ${\mathfrak{m}}^{t}\subseteq Ry$ and $|\Bbb{I}(Ry)|\geq 3$. \\
Suppose that  $|\Bbb{I}(Ry)|=3$. Since $Ry\cong R/{\rm Ann}(y)$
and $Ry$ has only one
 nonzero proper $R$-submodule,  ${\mathfrak{m}}/{\rm{Ann}}(y)$ is the only  nonzero proper ideal of $R/{\rm Ann}(y)$.
If ${\mathfrak{m}}^{2}\nsubseteq {\rm Ann}(y)$, then ${\rm
Ann}(y)+{\mathfrak{m}}^{2}={\mathfrak{m}}$, and thus by Nakayama's
lemma, ${\rm Ann}(y)={\mathfrak{m}}$, a contradiction. Therefore,
${\mathfrak{m}}^{2}\subseteq {\rm Ann}(y)$. So
${\mathfrak{m}}^{t-1}\subseteq {\rm Ann}(y)$. Hence,
${\mathfrak{m}}^{t-1}{\mathfrak{m}}=(0)$
(${\mathfrak{m}}^{t-1}(Ry)=(0)$ and
${\mathfrak{m}}^{t-1}(Ry_{1}+...+Ry_{n-1})=(0)$), a contradiction.
Therefore,
$|\Bbb{I}(Ry)|\geq 4$. \\
Assume $|\Bbb{I}(Ry)|<\infty$. Since $Ry\cong R/{\rm Ann}(y)$,
${\rm v.dim}_{R/{\mathfrak{m}}}({\rm
Ann}(y)+{\mathfrak{m}}^{2})/{\mathfrak{m}}^{2}=n-1$. Thus
$|\Bbb{I}({\rm Ann}(y))|=\infty$
 and $K_{|\Bbb{I}({\rm Ann}(y))|,|\Bbb{I}(Ry)|}$ is a subgraph of $\Bbb{AG}(R)$.
 So by the formula for genus of complete bipartite graphs, $\gamma(\Bbb{AG}(R))=\infty$.\\
 So we now assume that $|\Bbb{I}(Ry)|=\infty$.
 Suppose  that ${\rm v.dim}_{R/{\mathfrak{m}}}{\mathfrak{m}}^{2}/{\mathfrak{m}}^{3}\geq
 2$.  If $|\Bbb{I}({\rm Ann}(y))|\geq 4$, then $K_{|\Bbb{I}(Ry)|,3}$ is a subgraph of
 $\Bbb{AG}(R)$ and so by the formula for genus of complete bipartite
 graphs,
  $\gamma(\Bbb{AG}{R})=\infty$, a contradiction. We may assume that $|\Bbb{I}({\rm Ann}(y))|=3$.
  ${\rm Ann}(y)\nsubseteq {\mathfrak{m}}^{3}$ (since ${\mathfrak{m}}^{2}y\cong {\mathfrak{m}}^{2}/({\rm Ann}(y)\cap {\mathfrak{m}}^{2})$, there exist
finitely many
 ideals between ${\rm Ann}(y)$ and ${\mathfrak{m}}^{2}$). Therefore,  there exists $z\in {\rm Ann}(y)\setminus
 {\mathfrak{m}}^{3}$.  Since $|\Bbb{I}({\rm Ann}(y))|=3$, $Rz={\rm Ann}(y)$. Since $Rz\cong R/{\rm Ann}(z)$ and
$Rz$ has only one
 nonzero proper $R$-submodule,  ${\mathfrak{m}}/{\rm{Ann}}(z)$ is the only  nonzero proper ideal of $R/{\rm Ann}(z)$.
If ${\mathfrak{m}}^{2}\nsubseteq {\rm Ann}(z)$, then ${\rm
Ann}(z)+{\mathfrak{m}}^{2}={\mathfrak{m}}$, and thus by Nakayama's
lemma, ${\rm Ann}(z)={\mathfrak{m}}$, a contradiction. Therefore,
${\mathfrak{m}}^{2}\subseteq {\rm Ann}(z)={\rm Ann}({\rm
Ann}(y))$. Since ${\mathfrak{m}}^{2}{\rm Ann}(y)=(0)$,
${\mathfrak{m}}^{2}{\mathfrak{m}}^{t-1}=(0)$ and
 ${\mathfrak{m}}^{t}{\mathfrak{m}}^{2}=(0)$, $K_{|\Bbb{I}({\mathfrak{m}}^{2})|,3}$ is a subgraph of
 $\Bbb{AG}(R)$. So by the formula for genus of complete bipartite
 graphs,
  $\gamma(\Bbb{AG}(R))=\infty$, a contradiction. Therefore,
 ${\rm v.dim}_{R/{\mathfrak{m}}}{\mathfrak{m}}^{2}/{\mathfrak{m}}^{3}=1$ and by Lemma
 \ref{2.3},
 $|\Bbb{I}({\mathfrak{m}}^{2})|<\infty$.
 Since $my\cong {\mathfrak{m}}/{\rm Ann}(y)$ and $|\Bbb{I}({\mathfrak{m}}y)|<\infty$,
 ${\rm v.dim}_{R/{\mathfrak{m}}}({\rm Ann}(y)+{\mathfrak{m}}^{2})/{\mathfrak{m}}^{2}=n-1$. So $|\Bbb{I}({\rm Ann}(y))|=\infty$. Hence, $K_{|\Bbb{I}({\rm Ann}(y))|, |\Bbb{I}(Ry)|}$ is a subgraph
 of $\Bbb{AG}(R)$ and by the formula for genus of complete bipartite
 graphs,
  $\gamma(\Bbb{AG}(R))=\infty$, a contradiction. Thus ${\rm{v.dim}}_{R/{\mathfrak{m}}}{\mathfrak{m}}/{\mathfrak{m}}^{2}\leq 2$.\\
 Also, if  ${\rm v.dim}_{R/{\mathfrak{m}}}{\mathfrak{m}}/{\mathfrak{m}}^{2}=1$, then  by Lemma
\ref{2.3}, $R$ has finitely many ideals.\hfill $\square$
 \end{proof}

\begin{llem}\label{2.6}  Let $R=\prod_{i\in I}R_i$ be a product of nonzero rings $R_i$ $(i \in
I)$. If  $\gamma(\Bbb{AG}(R)) < \infty$, then $I$ is finite and
for each $i\in I$, $\gamma(\Bbb{AG}(R_i)) < \infty$. Consequently,
if  $R$ is  Artinian, then  $\gamma(\Bbb{AG}(R)) < \infty$ if and
only if $\gamma(\Bbb{AG}(R_{i})) < \infty$ for every $i\in I$.
\end{llem}

\begin{proof} Suppose, contrary to our claim, that $|I|=\infty$.
Let $I_{1}=(R_{i_{1}}\times R_{i_{2}}\times 0 \times ...\times 0)$
and $I_{2}=(0\times 0\times R_{i_{3}}\times R_{i_{4}}\times ...)$.
Since $I_{1}I_{2}=(0)$, $K_{|\Bbb{I}(I_{1})|,|\Bbb{I}(I_{2})|}$ is
a subgraph of $\Bbb{AG}(R)$. Since $|\Bbb{I}(I_{1})|\geq 4$ and
$|\Bbb{I}(I_{2})|=\infty$, by the formula for genus of complete
bipartite graphs, $\gamma(\Bbb{AG}(R)) = \infty$, a contradiction.
Therefore,  $|I|<\infty$ and since $\Bbb{AG}(R_{i})$ is a subgraph
of $\Bbb{AG}(R)$ for every $i\in I$ and $\gamma(\Bbb{AG}(R)) <
\infty$, we conclude that
 $\gamma(\Bbb{AG}(R_{i}))<\infty$ for every $i\in I$.\\
Now,   suppose that $R$ is an Artinian ring.  It is well known
that $R \cong R_{1} \times \ldots \times R_{n}$ for some positive
integer $n$, where every $R_{i}$  $(i=1,2,...,n)$ is an Artinian
 local ring. If $n=1$, then the proof is immediate. Now, we may
assume that $n\geq 2$. If every $R_{i}$ has finitely many ideals,
then $R$ has finitely many ideals and therefore
$\gamma(\Bbb{AG}(R))<\infty$.  Without loss of generality, we can
assume that $R_{1}$ has infinitely many ideals. Let
$I_{1}=(0\times R_{2}\times 0 \times ...\times 0)$, $I_{2}=({\rm
Ann}({\mathfrak{m}}_{1})\times 0\times ...\times 0)$ and
$I_{3}=({\rm Ann}({\mathfrak{m}}_{1})\times R_{2}\times 0\times
...\times 0)$. Then for every ideal $J$ of $R_{1}$, we have
 $I_{i}(J\times 0\times...\times 0)=(0)$. Therefore,  $K_{|\Bbb{I}({\mathfrak{m}}_{1})|-1,3}$ is a subgraph
of $\Bbb{AG}(R)$. Since $|\Bbb{I}({\mathfrak{m}}_{1})|=\infty$, by
the formula for genus of complete bipartite graphs,
$\gamma(\Bbb{AG}(R))=\infty$, a contradiction. Therefore,  every
$R_{i}$ has finitely many ideals and $\gamma(\Bbb{AG}(R))<\infty$.
\hfill $\square$
\end{proof}

\begin{ttheo}\label{2.7}
 Let $R$ be an Artinian ring. If
$\gamma(\Bbb{AG}(R)) < \infty$, then $R$ has finitely many ideals
or $(R,{\mathfrak{m}})$ is a Gorenstein ring with ${\rm
v.dim}_{R/{\mathfrak{m}}}{\mathfrak{m}}/{\mathfrak{m}}^{2}=2$.
\end{ttheo}

\begin{proof}
 Let $R$ be an Artinian ring. It is well known that $R
\cong R_{1} \times \ldots\times R_{n}$ for some positive integer
$n$, where every $R_{i}$  $(i=1,2,...,n)$ is an Artinian
 local ring and addition and multiplication in the
product are defined component wise. If $R$ is a local ring, then
by Theorem \ref{2.5}, $R$ has finitely many ideals or ${\rm
v.dim}_{R/{\mathfrak{m}}}{\mathfrak{m}}/{\mathfrak{m}}^{2}=2$. If
$n\geq 2$, then as in the proof of Lemma \ref{2.6}, every $R_{i}$
has finitely many ideals and therefore $R$ has finitely many
ideals. \hfill $\square$
\end{proof}

\begin{llem}\label{2.8}
 Let $(R,\mathfrak{m})$ be a Gorenstein ring, $k$ be a positive integer  and $q$ be a prime power.
 If ${\mathfrak{m}}^{2} \neq (0)$, ${\mathfrak{m}}^{3} = (0)$, $|R/{\mathfrak{m}}| = q$ and ${\it {\rm
v.dim}_{R/{\mathfrak{m}}}}({\mathfrak{m}}/{\mathfrak{m}}^{2})=k
>6$, then $\Bbb{AG}(R)$ contains a copy of $K_{k-6,3}$.
\end{llem}

\begin{proof}
Since ${\it {\rm v.dim}_{R/{\mathfrak{m}}}}
({\mathfrak{m}}/{\mathfrak{m}}^{2}) = k$ and
$|R|=|R/{\mathfrak{m}}||{\mathfrak{m}}/{\mathfrak{m}}^{2}||{\mathfrak{m}}^{2}|$,
we conclude that  $|R| = q^{k+2}$. Let
$\{x_{1}+{\mathfrak{m}}^{2},...,x_{k}+{\mathfrak{m}}^{2}\}$ be a
basis for ${\mathfrak{m}}/{\mathfrak{m}}^{2}$ over
$R/{\mathfrak{m}}$. We have $R/{\rm Ann}(x_{1}) \cong Rx_{1}$ thus
$|R| = |Rx_{1}||{\rm Ann}(x_{1})|$. Since ${\mathfrak{m}}^{2}
\subseteq Rx_{1} $, $|Rx_{1}| = |{\mathfrak{m}}^{2}||q|=q^{2}$ and
$|R| = q^{2} |{\rm Ann}(x_{1})|$. Therefore,  $|{\rm Ann}(x_{1})|
= q^{k}$ and ${\it {\rm v.dim}_{R/{\mathfrak{m}}}} ({\rm
Ann}(x_{1})/{\mathfrak{m}}^{2})= k-2$. Similarly ${\it {\rm
v.dim}_{R/{\mathfrak{m}}}} ({\rm Ann}(x_{2})/{\mathfrak{m}}^{2})=
k-2$. Thus ${\it {\rm v.dim}_{R/{\mathfrak{m}}}}(({\rm
Ann}(x_{1})/{\mathfrak{m}}^{2}) \cap ({\rm
Ann}(x_{2})/{\mathfrak{m}}^{2})) \geq k-4$. Let
$\{x_{i_{1}}+{\mathfrak{m}}^{2},...,x_{i_{n}}+{\mathfrak{m}}^{2}\}$
be a basis of $({\rm Ann}(x_{1})/{\mathfrak{m}}^{2}) \cap ({\rm
Ann}(x_{2})/{\mathfrak{m}}^{2})$. Suppose that $V_{1}=\{
Rx_{1},Rx_{2},{\mathfrak{m}}^{2}\}$ and
$V_{2}=\{Rx_{i_{1}},...,R_{x_{i_{n}}}\}$. Since
$K_{|V_{1}|,|V_{2}|}$ is a subgraph of $\Bbb{AG}(R)$, $K_{k-6,3}$
is a subgraph of  $\Bbb{AG}(R)$. \hfill $\square$
\end{proof}

\begin{llem}\label{2.9}
 Let $(R,\mathfrak{m})$ be an Artinian local ring with $|R/{\mathfrak{m}}|=q$
such that $R$ has finitely many ideals. Then
$|{\mathfrak{m}}^{i}|\leq q^{\Bbb{I}({\mathfrak{m}}^{i})}$ for all
$i$ and $|R|\leq q^{\Bbb{I}(R)}$
\end{llem}

\begin{proof} If $R$ is a field, then the proof is
straightforward. Thus we can assume that $R$ is not a field.
 Since $R$ is an Artinian ring, there exists positive
integer $t$ such that ${\mathfrak{m}}^{t}\neq (0)$ and
${\mathfrak{m}}^{t+1} =(0)$. Let ${\it {\rm
v.dim}_{R/{\mathfrak{m}}}}({\mathfrak{m}}^{i}/{\mathfrak{m}}^{i+1})=k_{i}$
where $1 \leq i\leq t$. Since
$|{\mathfrak{m}}^{i}|=|{\mathfrak{m}}^{i}/{\mathfrak{m}}^{i+1}||{\mathfrak{m}}^{i+1}/{\mathfrak{m}}^{i+2}|...|{\mathfrak{m}}^{t}|$
and $k_{i}+...+k_{t} \leq |\Bbb{I}({\mathfrak{m}}^{i})|$,
$|{\mathfrak{m}}^{i}|=q^{k_{i}+...+k_{t}} \leq
q^{\Bbb{I}({\mathfrak{m}}^{i})}$. Also, $|R|=
|R/{\mathfrak{m}}||m|=q^{k_{1}+...+k_{t}+1} \leq q^{\Bbb{I}(R)}$.
\hfill $\square$
\end{proof}

We recall that in \cite[Theorem 2]{wick2}, it was shown that for
every positive integer $g$ there exist finitely many finite rings
with $\gamma(\Gamma(R))=g$. But the result is not  true when we
replace "$\gamma(\Gamma(R))$" with "$\gamma(\Bbb{AG}(R))$". In
 fact, for every finite field $\Bbb{F}$, the ring  $R=\Bbb{F}\times \Bbb{F}\times \Bbb{F}\times
 \Bbb{F}$  has a toroidal  annihilating-ideal graph (since
  $\gamma(\Bbb{AG}(R))=\gamma(\Bbb{AG}(\Bbb{Z}_2\times\Bbb{Z}_2\times\Bbb{Z}_2\times\Bbb{Z}_2))=1)$. Next, we proceed to show that for
 every integers $q>0$ and $g \geq 0$,
there are finitely many Artinian  rings $R$, such that
$\gamma(\Bbb{AG}(R))=g$ and $|R/{\mathfrak{m}}| \leq q$ for every
 maximal ideal ${\mathfrak{m}}$ of $R$. We first prove a reduced form of the
 result.

 \begin{ttheo}\label{2.10}
 For a prime power $q$ and an integer $g \geq 0$,
there are finitely many Artinian local rings $(R,\mathfrak{m})$
satisfying  the following conditions:\\
 {\rm (1)} $\gamma(\Bbb{AG}(R)) = g$,
\\
{\rm (2)} $|R/{\mathfrak{m}}| = q$.
\end{ttheo}

\begin{proof}
 Suppose that  $R$ is  an Artinian local ring with
$|R/{\mathfrak{m}}|=q$ and $\gamma(\Bbb{AG}(R))=g$. Since there
exists only one field of order $q$, we can assume that $R$ is not
a field. Thus there exists a positive integer $t$ such that
${\mathfrak{m}}^{t}\neq (0)$, ${\mathfrak{m}}^{t+1} = (0)$. Since
$R$ is an Artinian ring and $|R/{\mathfrak{m}}|=q$, $R$ is a
finite
ring. It is sufficient to show that $|R|$ is bounded by a constant depending only on $g$ and $q$.
 The proof now proceeds by cases:\\
{\it Case} 1: ${\it {\rm v.dim}_{R/{\mathfrak{m}}}} ({\rm Ann}
({\mathfrak{m}})) \geq 2$. Since ${\rm Ann}({\mathfrak{m}})
{\mathfrak{m}}=(0)$, $K_{|\Bbb{I}(R)|- 5, 3}$ is a subgraph of
$\Bbb{AG}(R)$. By the formula for the genus of complete bipartite
graphs, $\lceil(|\Bbb{I}(R)|-7)/4 \rceil \leq g$ thus
$|\Bbb{I}(R)| \leq 4g+7$ and therefore by Lemma \ref{2.9},
$|R| \leq q^{4g+7}$.\\
{\it Case} 2: ${{\rm v.dim}_{R/{\mathfrak{m}}}} ({\rm Ann}
({\mathfrak{m}}))=1$.\\
{\it Subcase} 1:  $t=1$, i.e.,  ${\mathfrak{m}}^{2} = 0$. Since
$|R|=|R/{\mathfrak{m}}||{\mathfrak{m}}|$  and ${\it {\rm
v.dim}_{R/{\mathfrak{m}}}}{\mathfrak{m}}=1$, we conclude that
$|R|=q^{2}$.\\
{\it Subcase} 2:  $t=2$, i.e., ${\mathfrak{m}}^{2}\neq (0)$ and
${\mathfrak{m}}^{3}=(0)$. Let
 ${\it {\rm v.dim}_{R/{\mathfrak{m}}}}
({\mathfrak{m}}/{\mathfrak{m}}^{2}) = k$. Then  by Lemma
\ref{2.8}, $\Bbb{AG} (R)$ contains a copy of $K_{k-6,3}$.  By the
formula for the genus of complete bipartite graphs, $\lceil(k-8)/4
\rceil \leq g$ thus $|k| \leq 4g+8$. Since
$|R|=|R/{\mathfrak{m}}||{\mathfrak{m}}/{\mathfrak{m}}^{2}||{\mathfrak{m}}^{2}|$
and ${\it {\rm
v.dim}_{R/{\mathfrak{m}}}}{\mathfrak{m}}^{2}=1$, we conclude that  $|R|\leq q^{2+4g+8}$.\\
{\it Subcase} 3:   $t \geq 3$. The proof will be divided into
three steps.\\
{\it Step} 1: We prove that $|{\mathfrak{m}}^{3}|$ is bounded by a
constant depending only on $g$ and  $q$. If $t= 3,4,5$, then
$({\mathfrak{m}}^{3})^{2}=(0)$. Therefore,
$K_{|\Bbb{I}({\mathfrak{m}}^{3})|-1}$ is a subgraph of $\Bbb{AG}
(R)$ and by the formula for the genus of complete graphs, $\lceil
(|\Bbb{I}({\mathfrak{m}}^{3})|-5)/12 \rceil \leq g$ thus
$|\Bbb{I}({\mathfrak{m}}^{3})| \leq 12g+5$ and by Lemma \ref{2.9}, $|{\mathfrak{m}}^{3}| \leq q^{12g+5}$.\\
If $t \geq 5$ then
${\mathfrak{m}}^{t}{\mathfrak{m}}^{3}=(0),{\mathfrak{m}}^{t-1}{\mathfrak{m}}^{3}=(0),{\mathfrak{m}}^{t-2}{\mathfrak{m}}^{3}=(0)$.
Therefore,  $K_{|\Bbb{I}({\mathfrak{m}}^{3})|-4,3}$ is a subgraph
of $\Bbb{AG}(R)$ and by the formula for the genus of complete
bipartite graphs, $\lceil(|I({\mathfrak{m}}^{3})|-5)/4 \rceil \leq
g$. Thus $|\Bbb{I}({\mathfrak{m}}^{3})| \leq 4g+5$ and by Lemma
\ref{2.9},
 $|{\mathfrak{m}}^{3}| \leq q^{4g+5}$.\\
{\it Step} 2:  We prove that $|{\mathfrak{m}}^{2}|$ is bounded by
a constant depending only on $g$ and  $q$. If ${\it {\rm
v.dim}_{R/{\mathfrak{m}}}} ({\mathfrak{m}}^{2}/{\mathfrak{m}}^{3})
= 1$, then $|{\mathfrak{m}}^{2}|=q|{\mathfrak{m}}^{3}|$. Now,  by
Step 1,  we conclude that $|{\mathfrak{m}}^{2}|$ is also bounded
by a constant depending only on $g$ and  $q$. Suppose that ${\it
{\rm v.dim}_{R/{\mathfrak{m}}}}
({\mathfrak{m}}^{2}/{\mathfrak{m}}^{3}) = k\geq 2$.
 Let $\{x_{1}+{\mathfrak{m}}^{3},...,x_{k}+{\mathfrak{m}}^{3}\}$ be a basis for ${\mathfrak{m}}^{2}/{\mathfrak{m}}^{3}$
over $R/{\mathfrak{m}}$. Since ${\mathfrak{m}}^{t}\subsetneqq
Rx_{i}$,   $|\Bbb{I}(Rx_{i})|\geq 3$, for each $i$ $(1\leq i\leq
k)$.  If $|\Bbb{I}(Rx_{1})|=|\Bbb{I}(Rx_{2})|=3$, then  since
$Rx_{i}\cong R/{\rm Ann}(x_{i})$, $1\leq i\leq 2$, there is no
ideals between  ${\rm Ann}(x_{i})$, $1\leq i\leq 2$  and
${\mathfrak{m}}$. Now, if ${\mathfrak{m}}^{2}\nsubseteq{\rm
Ann}(x_{i})$, $1\leq i\leq 2$, then ${\rm
Ann}(x_{i})+{\mathfrak{m}}^{2}={\mathfrak{m}}$, $1\leq i\leq 2$
and by Nakayama's lemma, ${\rm Ann}(x_{i})={\mathfrak{m}}$, $1\leq
i\leq 2$, a contradiction. Therefore,  ${\mathfrak{m}}^{2}
\subseteq {\rm Ann}(x_{i})$, $1\leq i\leq 2$. Since
$({\mathfrak{m}}^{2})(Rx_{i})=(0)$, $1\leq i\leq 2$ and
${\mathfrak{m}}^{t}{\mathfrak{m}}^{2}=(0)$,
$K_{|\Bbb{I}({\mathfrak{m}}^{2})|-2,3}$ is a subgraph of $
\Bbb{AG}(R)$. By the formula for genus of complete bipartite
graphs, $\lceil (|\Bbb{I}({\mathfrak{m}}^{2})|-4)/4\rceil \leq g$,
thus $|\Bbb{I}({\mathfrak{m}}^{2})|\leq 4g+4$ and by Lemma
\ref{2.9}, $|{\mathfrak{m}}^{2}| \leq q^{4g+4}$.  Without loss of
generality we can assume that $|\Bbb{I}(Rx_{1})|\geq 4$. Then
$K_{|\Bbb{I}({\rm Ann}(x_{1}))|-4,3}$ is a subgraph of $
\Bbb{AG}(R)$. By the formula for genus of complete bipartite
graphs $\lceil (|\Bbb{I}({\rm Ann}(x_{1}))|-6)/4\rceil \leq g$,
thus $|\Bbb{I}({\rm Ann}(x_{1}))|\leq 4g+6$. Since $|\Bbb{I}({\rm
Ann}(x_{1}) \cap {\mathfrak{m}}^{2})|\leq |\Bbb{I}({\rm
Ann}(x_{1})|$ and dimension of every vector space less than or
equal to cardinality of each its generating set, we conclude that
${\it {\rm v.dim}_{R/{\mathfrak{m}}}}(({\rm Ann}(x_{1}) \cap
{\mathfrak{m}}^{2}) + {\mathfrak{m}}^{3})/{\mathfrak{m}}^{3} \leq
|\Bbb{I}({\rm Ann}(x_{1})|$. Let  $\theta : {\mathfrak{m}}^{2}
\rightarrow {\mathfrak{m}}^{2}x_{1}$ be the  group homomorphism
defined by multiplication by $x_{1}$. Then we have
$|{\mathfrak{m}}^{2}|/|{\mathfrak{m}}^{2}\cap {\rm Ann}(x_{1})| =
|{\mathfrak{m}}^{2}x|$. It follows that $|{\mathfrak{m}}^{2}|\leq
|{\mathfrak{m}}^{3}||{\mathfrak{m}}^{2}\cap {\rm Ann}(x_{1})| \leq
|{\mathfrak{m}}^{3}||{\mathfrak{m}}^{2}\cap {\rm Ann}(x_{1}) +
{\mathfrak{m}}^{3}| \leq |{\mathfrak{m}}^{3}|^{2} q^{|\Bbb{I}({\rm
Ann}(x_{1}))|} \leq |{\mathfrak{m}}^{3}|^{2} q^{4g+6} $.  Now,  by
Step 1, we conclude that $|{\mathfrak{m}}^{2}|$ is also bounded by a constant depending only on   $g$ and $q$.\\
{\it Step} 3: Now,  we show  that $|{\mathfrak{m}}|$ is bounded by
a constant depending only on $g$ and  $q$. We can assume that
${\it {\rm v.dim}_{R/{\mathfrak{m}}}}
({\mathfrak{m}}/{\mathfrak{m}}^{2})=k\geq 5$.
 Let $\{x_{1}+{\mathfrak{m}}^{2},...,x_{k}+{\mathfrak{m}}^{2}\}$ be a basis for ${\mathfrak{m}}/{\mathfrak{m}}^{2}$
over $R/{\mathfrak{m}}$. Since ${\mathfrak{m}}^{t}\subsetneqq
Rx_{i}$,   $|\Bbb{I}(Rx_{i})|\geq 3$, for each $i$ $(1\leq i\leq
k)$.  If $|\Bbb{I}(Rx_{1})|=|\Bbb{I}(Rx_{2})|=3$, then  since
$Rx_{i}\cong R/{\rm Ann}(x_{i})$, $1\leq i\leq 2$, there is no
ideals between  ${\rm Ann}(x_{i})$, $1\leq i\leq 2$ and
${\mathfrak{m}}$. If  ${\mathfrak{m}}^{2}\nsubseteq{\rm
Ann}(x_{i})$, for  $i=1$ or $2$,  then ${\rm
Ann}(x_{i})+{\mathfrak{m}}^{2}={\mathfrak{m}}$, and so by
Nakayama's lemma,  ${\rm Ann}(x_{i})={\mathfrak{m}}$. Since ${{\rm
v.dim}_{R/{\mathfrak{m}}}} ({\rm Ann} ({\mathfrak{m}}))=1$ and
${\mathfrak{m}}^t{\mathfrak{m}}=(0)$,
${\rm{Ann}}({\mathfrak{m}})={\mathfrak{m}}^t$. Thus $Rx_i\subseteq
{\mathfrak{m}}^t\subseteq {\mathfrak{m}}^2$, a contradiction.
Therefore,  ${\mathfrak{m}}^{2} \subseteq {\rm Ann}(x_{1})\cap
{\rm Ann}(x_{2})$. If ${\it {\rm v.dim}_{R/{\mathfrak{m}}}} ({\rm
Ann}(x_{i})/{\mathfrak{m}}^{2}) \leq k-2$ for $i=1$ or $2$, then
there exists proper ideal $I \subseteq R$ such that
${\rm{Ann}}(x_{i})/{\mathfrak{m}}^{2}\subsetneqq
I/{\mathfrak{m}}^{2}
\subsetneqq{\mathfrak{m}}/{\mathfrak{m}}^{2}$,  and hence  ${\rm
Ann}(x_{i})\subsetneqq I\subsetneqq {\mathfrak{m}}$, a
contradiction. Thus ${\it {\rm v.dim}_{R/{\mathfrak{m}}}} ({\rm
Ann}(x_{i})/{\mathfrak{m}}^{2}) \geq k-1$ for $i=1,~ 2$. It
follows that  ${\it {\rm v.dim}_{R/{\mathfrak{m}}}} ({\rm
Ann}(x_{1})\cap {\rm Ann}(x_{2})) \geq k-2$. Since  $Rx_{1}({\rm
Ann}(x_{1})\cap {\rm Ann}(x_{2})) = (0)$, $Rx_{2}({\rm
Ann}(x_{1})\cap {\rm Ann}(x_{2})) = (0)$ and
${\mathfrak{m}}^{t}({\rm Ann}(x_{1})\cap {\rm Ann}(x_{2})) = (0)$
and $|\Bbb{I}({\rm Ann}(x_{1})\cap {\rm
Ann}(x_{2}))\setminus\{(0), Rx_1, Rx_2\}|\geq k-5$,  we conclude
that $K_{k-5,3}$ is a subgraph of $\Bbb{AG}(R)$. Now,  by the
formula for genus of complete bipartite graphs, $k \leq 5g+6$
hence $|{\mathfrak{m}}| \leq |{\mathfrak{m}}^{2}|^{2}q^{5g+6}$.
 Thus by Step 2, we conclude that $|{\mathfrak{m}}|$ is also
bounded by a constant depending only on   $g$ and $q$.\\
Finally, without loss of generality we can assume that
$|\Bbb{I}(Rx_{1})| \geq4$. Thus  $K_{|\Bbb{I}({\rm
Ann}(x_{1}))|-4,3}$ is a subgraph of $\Bbb{AG}(R)$. By the formula
for genus of complete bipartite graphs, $\lceil |(\Bbb{I}({\rm
Ann}(x_{1}))|-6)/4 \rceil \leq g$ and hence $|\Bbb{I}({\rm
Ann}(x_{1}))|\leq 4g+6$. Since $|\Bbb{I}({\rm Ann}(x_{1}) \cap
{\mathfrak{m}}^{2})|\leq |\Bbb{I}({\rm Ann}(x_{1})|$ and the
dimension of every vector space less than or equal to cardinality
of each its generating set, we conclude that ${\it {\rm
v.dim}_{R/{\mathfrak{m}}}} ({\rm Ann}(x_{1})+{\mathfrak{m}}^{2})/
{\mathfrak{m}}^{2} \leq |\Bbb{I}({\rm Ann}(x_{1}))|$. Let
$\vartheta :{\mathfrak{m}} \rightarrow {\mathfrak{m}}x_{1}$ be the
group homomorphism defined by multiplication by $x_{1}$. Then
$|{\mathfrak{m}}|=|{\mathfrak{m}}x_{1}||{\rm Ann}(x_{1})| \leq
|{\mathfrak{m}}^{2}||{\rm Ann}(x_{1})+{\mathfrak{m}}^{2}| \leq
|{\mathfrak{m}}^{2}|^{2} q^{|\Bbb{I}({\rm Ann}(x_{1}))|} \leq
|{\mathfrak{m}}^{2}|^{2} q^{4g+6}$.  Now,  by Step 2, we conclude
that $|{\mathfrak{m}}|$ is also bounded by a constant depending
only on   $g$ and $q$. This finishes the proof, since
 $|R|=|R/{\mathfrak{m}}||{\mathfrak{m}}|$ and so   $|R|$ is bounded by a constant depending only on $g$
and $q$.\hfill$\square$
\end{proof}

\begin{ccoro}\label{2.11}
 For integers $q>0$ and $g \geq 0$,
there are finitely many Artinian  rings $R$ satisfying  the following conditions:\\
 {\rm (1)} $\gamma(\Bbb{AG}(R)) = g$,
\\
{\rm (2)} $|R/{\mathfrak{m}}| \leq q$  for any maximal ideal
${\mathfrak{m}}$ of $R$.
\end{ccoro}

\begin{proof}
Let $R$ be an Artinian ring with $\gamma(\Bbb{AG}(R)) = g$. It is
well known that $R \cong R_{1} \times \ldots \times R_{n}$ for
some positive integer $n$, where every $R_{i}$ $(i=1,2,...,n)$, is
an Artinian local ring. Since $|R_{i}/{\mathfrak{m}}_{i}| \leq q$,
every $R_{i}$ is a finite local ring. For ease of exposition, we
assume that $|R_{i}|\geq |R_{i+1}|$ for $i= 1,2,...,n-1$ as there
is no loss of generality with this
assumption.\\
It sufficient to show that $|R|$ is bounded by a constant
depending only on $g$ and $q$. Furthermore, $|R_{1}| \geq |R_{i}|$
for all $i$, it is enough to bound $R_{1}$ by a constant depending
only on $g$ and $q$. By a slight abuse of notation, we denote the
set $0 \times ... \times R_{i} \times ...\times 0$ by $R_{i}$, and
the set $R_{1} \times ...\times R_{i-1} \times 0 \times R_{i+1}
\times ... R_{n}$
by $\bar{R_{i}}$.\\
If $n \geq 3$, then $|\Bbb{I}(\overline{R_{1}^{*}})|\geq 3$. Thus
$\Bbb{AG}(R)$ contains a copy of $K_{|\Bbb{I}(R_{1})|-1,3}$. By
the formula for the genus of complete bipartite graphs, we have
$\lceil (|\Bbb{I}(R_{1})| - 3)/4\rceil \leq g$ so
$|\Bbb{I}(R_{1})| \leq 4g+3$ and thus by Lemma \ref{2.9},
$|R_{1}| \leq q^{4g+3}$.\\
Suppose that  $n=2$. Let $|\Bbb{I}(R_{2})|\geq 4$, this means
$\Bbb{AG}(R)$ contains a copy of
$K_{|\Bbb{I}(R_{1})|-1,3}$, which as above gives $|R_{1}|\leq q^{4g+3}$.\\
Let $|\Bbb{I}(R_{2})|\leq 3$. If $R_{1}$ be a field, then $|R_{1}|
\leq q$. Thus we may assume that $R_{1}$ is not a field. Let
${\mathfrak{m}} \neq (0)$ be the maximal ideal of $R_{1}$. Since
$R_{1}$ is an Artinian ring, ${\rm{Ann}}({\mathfrak{m}}) \neq
(0)$. Assume that  $I = ({\mathfrak{m}} , 0)$, $J_{1}= ({\rm Ann}
({\mathfrak{m}}) , 0)$, $J_{2} = ({\rm Ann} ({\mathfrak{m}}) ,
R_{2})$ and $J_{3} = (0 , R_{2})$. Then $IJ_{i} = (0)$ for $1 \leq
i\leq 3$. Therefore,  $\Bbb{AG}(R)$ contains a copy of
$K_{|\Bbb{I}(R_{1})|-2,3}$. Now,  by the formula for the genus of
complete bipartite graphs,
 $\lceil(|\Bbb{I}(R_{1})|-4)/4\rceil\leq g$ and so $|\Bbb{I}(R_{1})|\leq
 4g+4$.
Therefore,  by Lemma \ref{2.9}, $|R_{1}|\leq q^{4g+4}$.\\
Finally, suppose that $n=1$.
 Let $R$ be an Artinian local ring with $|R/{\mathfrak{m}}|=q$. Then by Theorem \ref{2.10},
 $|R|$ is bounded by a constant
depending only on $g$ and $q$. \hfill $\square$
 \end{proof}

 \section{\large\bf The Genus of  Annihilating-Ideal Graphs of Noetherian Rings}

 In this section, we investigate the genus of annihilating-ideal
 graphs of Noetherian rings.

Let $R$ be a ring. We say that the annihilating-ideal graph
$\Bbb{AG}(R)$ has assenting chain condition (ACC) (resp.,
descending chain condition (DCC)) on vertices if $R$ has ACC
(resp., DCC) on ${\Bbb{A}}(R)$.

First, we need the  following useful lemmas.

\begin{llem}\label{3.1}{\rm \cite[Theorem 1.1]{beh-rak1}} Let $R$ be a ring  that is not a domain.
 Then $\Bbb{AG}(R)$ has ACC (resp., DCC) on vertices if and only if
 $R$ is a Noetherian (resp., an Artinian) ring.
\end{llem}

 \begin{llem}\label{3.2}{\rm \cite[Proposition 1.7]{beh-rak1}} Let $R$ be a Noetherian ring. If all non-trivial ideals of $R$
 are vertices of $\Bbb{AG}(R)$, then $R$ has finitely many maximal ideals.
 \end{llem}

 By Lemma \ref{2.1}, any non-trivial ideal of an Artinian ring $R$ is a vertex of $\Bbb{AG}(R)$. Now, a natural question is posed: If $R$ is a Noetherian ring and
 all non-trivial ideals of $R$ are vertices of $\Bbb{AG}(R)$, then
 is $R$ an Artinian ring?  the answer of this question is negative (see  \cite[Example 17]{aal}).
 Next,  we have the following  result.

\begin{ppro}$\label{3.3}$ Let $(R, {\mathfrak{m}})$ be a Gorenstein ring. Then all non-trivial ideals of
$R$ are vertices of $\Bbb{AG}(R)$.
\end{ppro}
 \begin{proof} If $(R, {\mathfrak{m}})$ is  a Gorenstein ring, then
  ${\rm v.dim}_{R/{\mathfrak{m}}}({\rm
 Ann}({\mathfrak{m}}))=1$. It follows that  ${\rm
 Ann}({\mathfrak{m}})\neq (0)$ and so ${\mathfrak{m}}$ is a vertex
 of $\Bbb{AG}(R)$, i.e., all non-trivial ideals of
$R$ are vertices of $\Bbb{AG}(R)$. \hfill $\square$
 \end{proof}

\begin{ppro}\label{3.4}
 Let $R$ be a Noetherian local ring. If  $\gamma(\Bbb{AG}(R))<\infty$,
 then either $R$ is a domain or all non-trivial ideals of $R$ are vertices of $\Bbb{AG}(R)$.
 \end{ppro}

 \begin{proof} Let $R$ be a Noetherian ring and  $\gamma(\Bbb{AG}(R))<\infty$. We can assume that $R$
 is not a domain.  If $R$ is an Artinian ring, then by Lemma \ref{2.1}, every ideal of $R$ is a vertex of $\Bbb{AG}(R)$.
   Suppose that  $R$ is not an  Artinian  ring.
 By Lemma \ref{3.1},  there exists an infinite  descending chain $I_{1}\supsetneqq I_{2}\supsetneqq ...\supsetneqq
 I_{n}\supsetneqq\ldots$ in $\Bbb{A}(R)$. Suppose that  $P$ is a maximal element in $\Bbb{A}(R)$ such
  that $I_{1}\subseteq P$. Clearly $P$ is maximal among all annihilators of non-zero elements of
  $R$ and so  by \cite[Theorem 6]{kap}, $P$ is a prime ideal.   Also, $\{I_n~|~n\in\Bbb{N}\}\subseteq P$ and hence
$|\Bbb{I}(P)^*|=\infty$.
   Set  $I={\rm Ann}(P)$.  If $|\Bbb{I}(I)^*|\geq 3$, then $K_{|\Bbb{I}(I)^*|, |\Bbb{I}(P)^*|}$ is a
 subgraph of $\Bbb{AG}(R)$. Since $|\Bbb{I}(P)^*|=\infty$,
 $\gamma({\Bbb{AG}(R)})=\infty$, a contradiction. If
 $|\Bbb{I}(I)^*|=2$, then we must have $I=Rx$ for some $x\in R$, for otherwise, there are two nonzero
 distinct ideals  $Rz$, $Rx\subseteq I$, and so
 $|\Bbb{I}(I)^*|\geq 3$, a contradiction.
  Since ${\rm Ann}(I)={\rm Ann}(x)=P$,  $Rx\cong R/{\rm Ann}(x)$  and $Rx$ has only
  one nonzero  $R$-submodule, we conclude that
 there is no ideals between $P$ and ${\mathfrak{m}}$. Since  $P\subseteq P+{\mathfrak{m}}^{2}\subseteq
 {\mathfrak{m}}$,
 either $P+{\mathfrak{m}}^{2}=P$ or  $P+{\mathfrak{m}}^{2}={\mathfrak{m}}$.  If $P+{\mathfrak{m}}^{2}=P$, then
 ${\mathfrak{m}}^{2}\subseteq P$. It follows that $P={\mathfrak{m}}$. Also,
 if $P+{\mathfrak{m}}^{2}={\mathfrak{m}}$,  then    by Nakayama's lemma, $P={\mathfrak{m}}$. Thus
 $I=Rx$ is a simple $R$-module, contrary to
    $|\Bbb{I}(I)^*|=2$. Thus $|\Bbb{I}(I)^*|=1$, i.e.,  $I$ is a simple $R$-module. This implies that
     $P={\mathfrak{m}}$. Thus ${\mathfrak{m}}$ is a vertex of $\Bbb{AG}(R)$,
     i.e., all non-trivial ideals of $R$ are vertices of $\Bbb{AG}(R)$. \hfill $\square$
 \end{proof}

 \begin{ttheo}\label{3.5}
 Let $R$ be a Noetherian ring such that all non-trivial ideals of $R$ are vertices of $\Bbb{AG}(R)$.
 If  $\gamma(\Bbb{AG}(R))<\infty$,
 then either $R$ is a Gorenstein ring or  $R$ is an Artinian ring with finitely many ideals.
 \end{ttheo}

 \begin{proof} Suppose that $R$ is  a Noetherian ring,  all non-trivial ideals of $R$ are vertices
 of $\Bbb{AG}(R)$ and  $\gamma(\Bbb{AG}(R))<\infty$. Suppose that $R$ is a local
 ring with maximal ideal ${\mathfrak{m}}$.  We can assume that $R$ is not a Gorenstein
 ring.  Therefore ${\rm v.dim}_{R/{\mathfrak{m}}}({\rm Ann}({\mathfrak{m}}))\geq
 2$. It follows that  $|\Bbb{\rm{Ann}}(\mathfrak{m})^*|\geq 3$ and  ${\rm
 Ann}({\mathfrak{m}}){\mathfrak{m}}=(0)$. If  $|\Bbb{I}(R)|=\infty$, then  we conclude that
 $K_{|\Bbb{I}(R)|-5,3}$ is a subgraph of $\Bbb{AG}(R)$. Thus
 by the formula for the genus of complete bipartite graphs, we
 have $|\Bbb{I}(R)|=\infty$,  a contradiction. Thus $|\Bbb{I}(R)|<\infty$ $R$, i.e., $R$
  is an Artinian ring with finitely many ideals.\\
 Now, we can assume that $R$ is not a local
 ring. It is sufficient to shown
 that every prime ideal is a  maximal ideal.
   By Lemma \ref{3.2}, $R$ has finitely many maximal  ideals.
 Suppose that $k\geq 2$, ${\mathfrak{m}}_{1},...,{\mathfrak{m}}_{k}$ are all maximal ideals of
 $R$ and
 $J(R)={\mathfrak{m}}_{1}\cap\ldots\cap{\mathfrak{m}}_{k}$.   Consider the following chain
 $$J(R)\supseteq J(R)^{2}\supseteq \ldots \supseteq J(R)^{n}\supseteq ...$$
Suppose that $J(R)^{i}\supsetneqq J(R)^{i+1}$ for each $i\geq 1$.
Then $|\Bbb{I}(J(R))|=\infty$.  Since every ideal of $R$ is a
vertex of $\Bbb{AG}(R)$, we conclude that  $I_{1}={\rm
Ann}({\mathfrak{m}}_{1})\neq (0)$ and $I_{2}={\rm
Ann}({\mathfrak{m}}_{2})\neq (0)$. Also,  $I_{1}\nsubseteq I_{2}$
and  $I_{2}\nsubseteq I_{1}$ (otherwise, either
$(0)=I_{1}({\mathfrak{m}}_{1}+{\mathfrak{m}}_{2})=I_{1}R$ or
$(0)=I_{2}({\mathfrak{m}}_{1}+{\mathfrak{m}}_{2})=I_{2}R$, a
contradiction). Clearly
$I_{1}J_{R}=I_{2}J(R)=(I_{1}+I_{2})J(R)=(0)$ and so
$K_{3,|\Bbb{I}(J(R))^*|}$ is a subgraph of $\Bbb{AG}(R)$. Since
$|\Bbb{I}(J(R))|=\infty$, $\gamma(\Bbb{AG}(R))=\infty$, a
contradiction.
 Therefore $J(R)^{n}=J(R)^{n+1}$ for some  $n\geq 1$. Thus by Nakayama's
 lemma,
 $J(R)^{n}=(0)$. Suppose that  $P$ is a prime ideal of $R$, since $J(R)^{n}=(0)\subseteq
 P$, there exists ${\mathfrak{m}}_{i}$, $1\leq i\leq n$ such that ${\mathfrak{m}}_{i}=P$.
  Thus every prime ideal of $R$ is a maximal ideal of $R$. Thus $R$ is an Artinian ring and by
   Theorem \ref{2.7}, $R$ has finitely many ideals.~\hfill $\square$
 \end{proof}

We conclude this paper with the following corollary.

\begin{ccoro}\label{3.6}
 Let $R$ be a non-domain Noetherian local ring. If  $\gamma(\Bbb{AG}(R))<\infty$,
 then either $R$ is  a Gorenstein ring or $R$ is an Artinian ring with finitely many ideals.
 \end{ccoro}
 \begin{proof}
 By Proposition 3.4 and Theorem 3.5 is clear.~\hfill $\square$
  \end{proof}


\begin{thebibliography}{1}
 \bibitem{aal} G. Aalipour, S. Akbari, M. Behboodi, R. Nikandish,
 M. J. Nikmehr and F. Shahsavari, The classification of the annihilating-ideal graph
 of a commutative ring, {\it Algebra Collocuium}, to appear.
 \vspace{-3mm}
 \bibitem{ak-moh} S. Akbari, A. Mohammadian,
On the zero-divisor graph of a commutative ring, {\it J. Algebra}
{\bf 274} (2004) 847-855.
\vspace{-3mm}
\bibitem{and-liv} D. F. Anderson, P.
S. Livingston, The zero-divisor graph of a commutative ring, {\it
J. Algebra} {\bf 217} (1999) 434-447.
 \vspace{-3mm}
\bibitem{and-nas} D. D. Anderson, M. Naseer,  Beck's coloring
of a commutative ring, {\it J. Algebra} {\bf 159} (1993)
500-514.\vspace{-3mm}
\bibitem{beck} I. Beck,  Coloring of commutative rings, {\it J.
Algebra} {\bf 116} (1988) 208-226. \vspace{-3mm}
\bibitem{beh-rak1} M. Behboodi, Z. Rakeei, The annihilating-ideal graph of
 commutative rings I, {\it J. Algebra Appl.} to appear.\vspace{-3mm}
\bibitem{beh-rak2} M. Behboodi, Z. Rakeei, The annihilating-ideal graph of
 commutative rings II, {\it J. Algebra Appl.} to appear.\vspace{-3mm}
\bibitem{hal}P. R. Halmos. {\it Linear Algebra problem Book}, The
  mathematical Association of America, 1995.\vspace{-3mm}
  \bibitem{kap} I. Kaplanski, {\it Commutative Rings}. rev. ed. Chicago:
  Univ. of Chicago Press, 1974.\vspace{-3mm}
\bibitem{ring1} G. Ringel, Der vollstandinge paare graph auf
nichtorientierbarenFlachen, {\it J. Reine Angew. Math.} {\bf 220}
(1975) 89-93.\vspace{-3mm}
\bibitem{ring2} G. Ringel, {\it Map Color Theorem},
Springer-verlag, New York, 1974.\vspace{-3mm}
\bibitem{ring-you} G. Ringel, J. W. T. Youngs, Solution of the heawood map-coloring problem,{\it Proc.
Nat. Acad. Sei. USA}, {\bf 60} (1968) 438–445.\vspace{-3mm}
\bibitem{wang} Hsin-Ju Wang, Zero-divisor graphs of genus one, {\it J. Algebra} {\bf 304} (2006) 666-678.\vspace{-3mm}
\bibitem{wick1}Cameron Wickham, Classification of Rings with
Genus One Zero-Divisor Graphs,
 {\it  Comm.  Algebra} {\bf 36} (2008) 325-345.\vspace{-3mm}
\bibitem{wick2} Cameron Wickham, Rings whose zero-divisor graphs
have
  positive genus, {\it J. Algebra} {\bf321} (2009) 377-383.\vspace{-3mm}
\bibitem{white} A. T. White, {\it Graphs, Groups and Surfaces},
North-Holland:
 Amsterdam, 1973.
 \end{thebibliography}
\end{document}